\documentclass[11pt]{amsart}
\usepackage{amsmath, amsthm, amssymb}
\usepackage{graphicx}
\usepackage[dvipsnames]{xcolor} 
\usepackage{srcltx} 
\usepackage{verbatim}
\usepackage{hyperref}
\hypersetup{
    colorlinks=true,
    linkcolor=blue,
    filecolor=magenta,      
    urlcolor=blue,
    citecolor = blue,
    }
\usepackage{tikz-cd}

\newtheorem{thm}{Theorem}[section]
\newcommand{\bt}{\begin{thm}}
\newcommand{\et}{\end{thm}}

\newtheorem{ex}[thm]{Example}

\newtheorem{cor}[thm]{Corollary}   
 
\newcommand{\bc}{\begin{cor}}
\newcommand{\ec}{\end{cor}}

\newtheorem{lem}[thm]{Lemma}   
\newcommand{\bl}{\begin{lem}}
\newcommand{\el}{\end{lem}}

\newtheorem{prop}[thm]{Proposition}
\newcommand{\bp}{\begin{prop}}
\newcommand{\ep}{\end{prop}}

\newtheorem{defn}[thm]{Definition}
\newcommand{\bd}{\begin{defn}}    
\newcommand{\ed}{\end{defn}}

\newtheorem{rmrk}[thm]{Remark}   

\newcommand{\br}{\begin{rmrk}}
\newcommand{\er}{\end{rmrk}}

\newtheorem*{*thm}{Theorem}

\newcommand{\be}{\begin{equation}}

\newcommand{\ee}{\end{equation}}

\newcommand{\injrad}{\textrm{injrad}}

\newcommand{\N}{\mathbb{N}}

\newcommand{\vare}{\varepsilon}

\newcommand{\diam}{\operatorname{Diam}}

\newcommand{\vol}{\operatorname{Vol}}

\newcommand{\lp}{\left (}
\newcommand{\rp}{\right )}

\newcommand{\Sp}{\mathbb{S}}      
\newcommand{\Tor}{\mathbb{T}}     

\def\XXint#1#2#3{{\setbox0=\hbox{$#1{#2#3}{\int}$ }
\vcenter{\hbox{$#2#3$ }}\kern-.6\wd0}}

\title[Sobolev Inequalities for Rough Riemannian Metrics]{Sobolev Inequalities Relating Rough Riemannian Metrics to their Distance Functions}

\author{Brian Allen}
\author{Lisa Naples}

\begin{document}

\maketitle

\begin{abstract}
In previous papers, the first named author and E. Bryden explored Sobolev like inequalities relating the norm of a Riemmanian metric $g_1$, in terms of a background Riemannian metric $g_0$, in order to estimate the distance function $d_{g_1}$. Although $|g_1|_{g_0}$ is not the gradient of the distance function $d_{g_1}$, it was shown that a Morrey and Sobolev inequalities hold in their appropriate domains. In this paper we see that these results can be extended to the case where $g_1$ is a Rough Riemannian metric, something roughly analogous to an $L^{\infty}$ metric. We show that the proof of the Morrey and Sobolev inequality go through, with appropriate adjustments, as well as convergence results. We also adapt previous examples from the literature to the rough Riemannian metric case which show the hypotheses from the main theorems are necessary.
\end{abstract}

\section{Introduction}

Sobolev inequalities relating Sobolev spaces $W^{k,p}$ to higher integrability in $L^q$ or higher regularity in $C^{k,\alpha}$ are an indispensable tool to studying existence and regularity of weak solutions to partial differential equations. If we consider two Riemannian metrics $g_0, g_1$ defined on $M^m$ with corresponding distance functions $d_{g_0}, d_{g_1}$ then one could hope to establish Sobolev like inequalities between $d_{g_1}$ and $|g_1|_{g_0}$, by imagining $|g_1|_{g_0}$ as the norm of the gradient of $d_{g_1}$. It was precisely this idea which was explored by the first named author and E. Bryden in a series of papers \cite{Allen-Bryden, Allen, Allen-Bryden-Second}. There we were able to establish Sobolev ($1< p < $m) and Morrey ($p > m$) inequalities where $g_1$ was assumed to be a continuous Riemannian metric and $g_0$ is assumed smooth. Two important questions can then immadietely be raised: 1. How far can one weaken the regularity of $g_1$? 2. How far can one weaken the regularity of $g_0$? 

In this paper we address the first question by proving versions of the previously established Sobolev and Morrey inequalities where $g_1$ is only assumed to be a rough Riemannian metric, which is morally like assuming $g_1$ is an $L^{\infty}$ metric, defined precisely in section \ref{sec-Background}.  Rough Riemannian metrics were defined and studied by L. Bandera and several collaborators in a series of papers \cite{BandaraBryan, BandaraHassan, BNR, BandaraRough}, establishing many important properties and demonstrating the utility of this class of geometric objects. More recently, the first named author, B. Falcao, H. Pacheco, and B. Sanchez \cite{LehmanUndergrad2}, a team of undergraduate students at Lehman College,  also studied convergence of sequences of rough Riemannian metrics, continuing the study of this important class of geometric objects. 

One of the main technical hurdles one must overcome when proving a Sobolev inequality relating $|g_1|_{g_0}$ to $d_{g_1}$ is how to relate distances, which come from minimizing lengths of curves, to integrals of $|g_1|_{g_0}$ over open sets in $M$. The main tool employed in \cite{Allen, Allen-Bryden-Second} was called symmetric families of curves connecting $x$ to $y$ for $x,y \in M$, depicted in figure \ref{fig:SymmetricFamilyCurves}. The symmetric family of curves has a length minimizing geodesic with respect to $g_0$, $\alpha$, at its core and uses Jacobi fields to then build an open set, foliated by curves connecting $x$ to $y$. We use the fact that $g_0$ is assumed to be smooth in order to control the geometry of the symmetric families of curves, making sure that they can be chosen in a uniform way for all $x,y \in M$. 

Given the uniform control from $g_0$ on the symmetric families of curves connecting $x$ to $y$ we can then move on to estimating $d_{g_1}$ in terms of $|g_1|_{g_0}$. When $g_1$ was assumed to be continuous, the main idea of the proofs was to start with lengths of curves integrated over the symmetric family of curves, thought of as a parametrization space foliated by these curves. This allows us to estimate the distance function $d_{g_1}$ in terms of an integral over the parameterization space of a family of curves. The main issue when adapting this proof to the case where $g_1$ is a rough Riemmanian metric is that their corresponding distance functions cannot be defined in terms of minimizing the length of curves. Instead, it can be defined as the supremum of $|w(x)-w(y)|$ for functions $w$ on $M$ so that $|\nabla w|_{g_1} \le 1$ a.e. Hence the main work of this paper is to adapt the proof of \cite{Allen, Allen-Bryden} to this new definition of distance function. This requires considering $w \circ \eta$ where $w$ is a function almost realizing the distance function $d_{g_1}$ and $\eta$ is a curve in the foliation of the symmetric family of curves. We show that the previous proofs can be adapted to the case of $w \circ \eta$ defined on symmetric families of curves.

\begin{figure}
 \label{fig:SymmetricFamilyCurves}
 \begin{tikzpicture}[scale=4]
  \draw[thick] (-1,1) -- (1,1);
  \draw[fill] (-1,1) circle [radius=0.025];
  \node[left, outer sep=2pt] at (-1,1) {$x$};
  \draw[fill] (1,1) circle [radius=0.025];
  \node[right, outer sep=2pt] at (1,1) {$y$};
  \node[right, outer sep=2pt] at (0,1-1/16) {$\alpha$};
   \draw plot [smooth] coordinates {(-1,1)(-.5,1.1)(0,1.15)(.5,1.1)(1,1)};
  \draw plot [smooth] coordinates {(-1,1)(-.5,1-.1)(0,1-.15)(.5,1-.1)(1,1)};
  \draw plot [smooth] coordinates {(-1,1)(-.5,1.2)(0,1.3)(.5,1.2)(1,1)};
  \draw plot [smooth] coordinates {(-1,1)(-.5,1-.2)(0,1-.3)(.5,1-.2)(1,1)};
  \draw plot [smooth] coordinates {(-1,1)(-.5,1.27)(0,1.4)(.5,1.27)(1,1)};
  \draw plot [smooth] coordinates {(-1,1)(-.5,1-.27)(0,1-.4)(.5,1-.27)(1,1)};
  \draw plot [smooth] coordinates {(-1,1)(-.5,1.35)(0,1.5)(.5,1.35)(1,1)};
  \draw plot [smooth] coordinates {(-1,1)(-.5,1-.35)(0,1-.5)(.5,1-.35)(1,1)};
 \end{tikzpicture} 
 \end{figure}

The first main theorem shows that the Theorem 1.1 of \cite{Allen-Bryden-Second} holds for rough Riemannian metrics. In particular a Sobolev inequality relating the $L^{\frac{p}{2}}$ norm of a rough Riemannian metric to the $L^q$ norm of its corresponding distance function.

\begin{thm}\label{Main Sobolev Ineq}
    Let $M^m$ be a smooth, connected, closed, and oriented manifold and $(M,g)$ be a fixed Riemannian manifold with metric $g_0$, and let $g_1$ be a rough Riemannian metric on $M$. Then, for $p < m$ we have
    \begin{equation}
        \|d_{g_1}(x,y)\|_{L^{q}(M\times M, g_0\oplus g_0)}\le C(M,g_0,q)\|g_1\|^{\frac{1}{2}}_{L^{\frac{p}{2}}(M,g_0)},
    \end{equation}
    with $q<\frac{mp}{m-p}$.
\end{thm}

Similarly, if we increase the assumed integrability of the norm of the Riemannian metric we obtain a Morrey like inequality.

\begin{thm}\label{MainThmEst}
Let $M^m$ be a smooth, connected, closed, and oriented manifold, $g_0$ a smooth Riemannian metric, and $g_1$ a rough Riemannian metric. If there exists  $p > m$ so that
 \begin{align}
 \quad \|g_1\|_{L_{g_0}^{\frac{p}{2}}(M)} \le C,
 \end{align}
  then 
\begin{align}
d_{g_1}(x,y) \le C'(M,g_0,C) d_{g_0}(x,y)^{\frac{p-m}{p}},\quad \forall x,y \in M.
\end{align}

\end{thm}

\begin{rmrk}
    By applying the Arzel\`{a}-Ascoli Theorem one can obtain compactness from the H\"{o}lder estimate in Theorem \ref{MainThmEst} in a standard way where one can see the details worked out in \cite{Allen}.
\end{rmrk}

We can adapt the Sobolev inequality of Theorem \ref{Main Sobolev Ineq} in order to prove a convergence result for sequences of Riemannian manifolds.

\begin{thm}\label{MainTheoremConv}
Let $M^m$ be a smooth, connected, closed, and oriented manifold, $g_0$ a smooth Riemannian manifold, and $g_j$ a sequence of rough Riemannian metrics. If
\begin{align}
    g_j(v,v) \ge \left(1-\frac{1}{j} \right) g_0(v,v), \quad \forall p\in M, v \in T_pM,
\end{align}
and
\begin{align}
    \|g_j-g_0\|_{L^{\frac{p}{2}}_{g_0}(M)}\rightarrow 0,\quad 1 \le p\le m,
\end{align}
then there exists a subsequence, $g_k$, and a measurable set $U \subset M$ so that $\vol_{g_0}(U)=\vol_{g_0}(M)$, $\mathrm{dim}_{\mathcal{H}_{g_0}}(M \setminus U)\le m-p$, and
\begin{align}
    d_{g_k}(x,y) \rightarrow d_{g_0}(x,y), \quad \forall x,y \in U.
\end{align}
Furthermore, for every $\varepsilon > 0$ we can find a measurable set $U_{\varepsilon}\subset M$ so that $\vol_{g_0}(M \setminus U_{\varepsilon})< \varepsilon$ and 
\begin{align}
    d_{g_k}(x,y) \rightarrow d_{g_0}(x,y),
\end{align}
uniformly $\forall x,y \in U_{\varepsilon}$.
\end{thm}

We believe that the results of this paper show that one should generally expect to be able to weaken similar convergence results, such as the VADB Theorem of the first named author, R. Perales, and C. Sormani \cite{ Allen-Perales-Sormani}, and future results of this type, to the case of rough Riemannian metrics as well. In a similar vein, one should be able to weaken the background metrics as well, as mentioned in question 2, but this is generally a much more delicate question. In particular, when attempting to answer this question one should not expect to be able to weaken the regularity of $g_0$ to anything even remotely as weak as rough Riemannian metrics or even continuous Riemannian metrics, because of the need to control the symmetric families of curves for all $x,y \in M$. If one wants to weaken the regularity of $g_0$ then one will have to keep an eye to the essential features of the symmetric families of curves which must be controlled to make the proofs work or come up with entirely different proof altogether.

In section \ref{sec-Background}, we review results from previous papers which will be needed here as well as the definition of rough Riemannian metrics and their properties. In section \ref{sec-Examples}, we adapt examples from the literature to the rough Riemannian case which show that the assumptions in the main theorems are necessary. In section \ref{sec-Poincare}, we collect the proofs of each of the main results.

\vspace{.5cm}

\textbf{AI Use Statement:} AI was not used in any way in the preperation of this document.

\section{Background}\label{sec-Background}

In this section we review the tools from the first named author's paper \cite{Allen} which we will use throughout the rest of the paper. 

\subsection{$L^p$ spaces for Riemannian metrics}
We start with a useful lemma for estimating a norm with respect to one Riemannian metric in terms of another.

\begin{lem}\label{NormComparisonLemma}
 For $v \in T_pM$ and $g_0,g_1$ Riemannian metrics on $M$ we have the inequality
 \begin{align}
 |v|_{g_1}\le |g_1|_{g_0}^{\frac{1}{2}}|v|_{g_0}.
 \end{align}
 \end{lem}
 We also review the definition of the $L^p$ norm of $g_1$ with respect to $g_0$,
 
 \begin{align}
     \|g_1\|_{L^p_{g_0}(M)}=\left(\int_M |g_1|_{g_0}^pdV_{g_0}\right)^{1/p},
 \end{align}
 where $dV_{g_0}$ is the volume form with respect to $g_0$.

 \subsection{Symmetric Families of Curves}
 
 We now give the definition of a symmetric family of curves, depicted in Figure \ref{fig:SymmetricFamilyCurves}, which foliates a region of full volume by curves connecting points $x,y \in M$.

\begin{defn}\label{SymFamilyofCurvesDef}
 Let $x,y \in M_0$ and $\alpha(t)$ be a length minimizing geodesic joining $x$ to $y$, with respect to $M_0$, of length $L$. By extending a distance $\tau \in (0,\varepsilon)$ in radial directions orthogonal to $\alpha'(t)$, which can be parameterized over $\Sp^{m-2}$, one obtains a tubular neighborhood $\alpha \subset \bar{U}_{\varepsilon} \subset M$ with coordinates $(\tau,t,\vec{s}) \in [0,\varepsilon]\times [0,L]\times\Sp^{m-2}$. For fixed $(\tau,\vec{s}) \in  [0,\varepsilon] \times \Sp^{m-2}$ we define a curve connecting $x$ and $y$ in coordinates by
 \begin{align}
 \gamma(\tau,t,\vec{s}) = (\tau L\sin\left(\frac{\pi}{L}t\right),t,\vec{s}).
 \end{align} 
 We define the \textbf{symmetric family of curves joining $x$ to $y$ of width $\varepsilon$} to be the set $U_{\varepsilon}$ foliated by the curves $\gamma(\tau,t,\vec{s})$ for $(\tau,t,\vec{s}) \in (0,\varepsilon)\times [0,L]\times\Sp^{m-2}$.

 Let $d\gamma^{-1}$ be the differential of the map $\gamma^{-1}:U_{\vare}\setminus\{x,y\}\rightarrow[0,\vare]\times(0,L)\times\Sp^{m-2}$ and  $[\gamma'(p)]^{\perp}\subset T_pM$ be the orthogonal subspace to $\gamma'(p)$. The \textbf{normal Jacobian} of $\gamma^{-1}$ at $x$ is the determinant of the restriction of $d\gamma^{-1}\big\vert_{x}$ to $[\gamma'(x)]^{\perp}$ denoted
 \begin{align}
    NJ\gamma^{-1}(x):= \det_{[\gamma'(x)]_{g_0}^{\perp}}\left(d\gamma^{-1}\big\vert_{x}\right).
 \end{align}
 \end{defn}

 We then review a Lemma which guarantees the existence of a normal cylindrical neighborhood around any geodesic with a uniform size which does not depend on the points $x,y \in M$ which the geodesic is connecting. 
 
 \begin{lem}\label{NorNeighRadiusEst}
 Let $M_0=(M,g_0)$ be a smooth Riemannian manifold defined on the closed manifold $M$. Let $K\in [0,\infty)$ be a bound on the absolute value of the sectional curvature of $M_0$ and $\injrad(M_0)$ be the injectivity radius. If  $x,y\in M_0$ and $\alpha(t)$ a length minimizing geodesic joining $x$ to $y$, parameterized by arc length, then there exists a $\bar{\varepsilon}(K,\injrad(M_0)) > 0$ so that for any $0 < \varepsilon < \bar{\varepsilon}$  a tubular neighborhood of width $\varepsilon$ exists, $\bar{U}_{\varepsilon}$.
 \end{lem}
 
 Lastly, we state a lemma which gives estimates on the geometry of the symmetric family of curves in terms of geometric bounds on the background Riemannian metric $g_0$.
 
 \begin{lem}\label{MainToolsForProof}
 Let $M_0=(M,g_0)$ be a smooth Riemannian manifold defined on the smooth, connected, closed manifold $M$. Assume  $x,y\in M_0$ and $\alpha(t)$ a length minimizing geodesic joining $x$ to $y$, parameterized by arc length. then there exists an $\varepsilon >0$ so that a symmetric family of curves joining $x$ to $y$ of width $\varepsilon$ exists, $U_{\varepsilon}$. If $K \in [0,\infty)$, is a bound on the absolute value of the sectional curvature, $I_0=\injrad(M_0)$ the injectivity radius, and $D_0=\diam(M_0)$ the diameter of $M_0$ then for $\bar{\varepsilon}=\bar{\varepsilon}(K, \injrad(M_0))$ from Lemma \ref{NorNeighRadiusEst} we have that for $0<\varepsilon \le \frac{\bar{\varepsilon}}{\diam(M_0)} $ a symmetric family of curves joining $x$ to $y$ of width $\varepsilon$ exists, $U_{\varepsilon}$,  and with $\gamma'=\frac{d\gamma}{dt}$ we find
 \begin{align}
 |\gamma'(\tau,t,\vec{s})|_{g_0} \le C(K,I_0,D_0), \quad (\tau,t,\vec{s}) \in [0,\varepsilon]\times [0,L]\times \Sp^{m-2}.
\end{align}  
We find the normal jacobian of $\gamma^{-1}$ to be estimated by
\begin{align}\label{Curve_Family_Coarea_Formula}
NJ\gamma^{-1}&\le C(M,g_0) \left(L\sin\left(\frac{\pi}{L}t\right) \right)^{1-m}.
\end{align}

 \end{lem}

 \subsection{Sobolev Inequalities for Smooth Riemannian metrics}
  In this subsection we review some results from \cite{Allen-Bryden-Second} which do not change in the present context. The first lemma is an estimate involving an application of the coarea formula to symmetric families of curves.
  
\begin{lem}\label{Coarea Formula to Potential Integral}Let $U_{\vare}$ be a symmetric family of curves connecting two points of $M$, say $x$ and $y$. Let $\Gamma_{\vare}=\left([0,\vare]\times\mathbb{S}^{m-2},d\tau^2+\tau^2g_{\Sp^{m-2}}\right)$, which parameterizes the curves foliating $U_{\vare}$. Let $\mu_{\vare}$ denote the volume measure for $\Gamma_{\vare}$. Then, for any non-negative measurable function, say $f$, we have
\begin{equation}
    \int_{\Gamma_{\vare}\times[0,L]}fdtd\mu_{\vare}\le C(M,g_{0})\int_{U_{\vare}} f(z)\left(d_{g_0}(x,z)^{1-m}+d_{g_{0}}(y,z)^{1-m}\right)dV_{g_{0}}(z)
\end{equation}

\end{lem}

We now observe an estimate of lengths of curves in a symmetric family of curves.

\begin{lem}\label{uniform distance estimate symmetric family of curves}
    Let  $(M,g_0)$ be a  compact Riemannian manifold. Then, for any $\vare>0$ there exists a $\delta(M,g_0,\vare):=\delta=>0$ such that if $U_{\delta}$ is a symmetric family of curves about a minimal length path connecting any two points, say $x$ to $y$, then for any curve in the family $U_{\delta}$, say $c$, we have
    \begin{equation}
        l(c)\leq d_{g_{0}}(x,y)+2\vare.
    \end{equation}
\end{lem}

Next a result relating integratbility of potentials.

\begin{prop}\label{PotentialOperatorEstimate}
    Consider a closed Riemannian manifold $(M,g_0)$ and let $f$ be an element of $L^{p}(dVol_{g_{0}})$. Then, the function
    \begin{equation}
        \left(V_{g_{0}}f\right)(x):=\int_{M}f(z)d_{g_{0}}(x,z)^{1-m}dV_{g_{0}}(z),
    \end{equation}
    satisfies the bound
    \begin{equation}
        \|V_{g_{0}}f\|_{L^q}\leq C(M,g_{0},q)\|f\|_{L^{p}}
    \end{equation}
    for 
    \begin{equation}
        q<\frac{mp}{m-p}.
    \end{equation}
\end{prop}

\begin{rmrk}\label{remark-TrivialTransformSubstitution}
     Notice by plugging $f=1$ into the right hand side of Lemma \ref{Coarea Formula to Potential Integral} one obtains
    \begin{align}
        \int_{M}\frac{dV_{g_0}(z)}{d_{g_{0}}(x,z)^{m-1}}+\int_{M}\frac{dV_{g_0}(z)}{d_{g_{0}}(y,z)^{m-1}}= k(x)+k(y)
    \end{align}
    where $|k|_{L^{\infty}(M,g_0)} \le C(M,g_0)$ since in polar coordinate around $x$ or $y$ we see that $r^{1-m}$ is integrable and bounded by a constant $C(M,g_0)$, and where the integral away from the polar coordinates is bounded by $C\vol(M,g_0)$ .
\end{rmrk}

We end with a proposition that helps us estimate the singular set in Theorem \ref{MainTheoremConv} which will be used in Section \ref{sec-Poincare}.

\begin{prop}\label{Lp Hausdorff dimension control}
    Let $\delta>0$, and let $f_{j}$ be a sequence of non-negative functions tending to zero in $L^{p}(g_{0})$.
    If $f_{j}$ is such that
    \begin{equation}
        \sum^{\infty}_{j=1}\|f_{j}\|^p_{L^{p}(dV_{g_{0}})}<\infty,
    \end{equation}
    and $B(\delta)$ is defined by 
    \begin{equation}
        B(\delta)=\left\lbrace x:\limsup_{j\rightarrow\infty}\int_M\frac{f_{j}(y)}{d_{g_{0}}(x,y)^{m-1}}dV_{g_{0}}(y)\ge\delta\right\rbrace.
    \end{equation}
    then we have
    \begin{equation}
        \mathrm{dim}_{H_{g_{0}}}\left(B(\delta)\right)\le m-p.
    \end{equation}
\end{prop}

\subsection{Rough Riemannian Metrics}

We begin by defining rough Riemannian metrics and recalling some properties established in \cite{BandaraHassan}. 

\begin{defn}
 We say that $g$ is a rough Riemannian metric if it is a measurable symmetric $2-$tensor on $M$ that satisfies the condition: for every $x \in M$, there exists a chart $(U,\varphi)$ and a constant $C_U\ge 1$ so that
 \begin{align}
     C_U^{-1} |v|_{\varphi^*\delta(y)} \le |v|_{g(y)} \le C_U|v|_{\varphi^*\delta(y)},
 \end{align}
for $v \in T_yU$ and for almost every $y \in U$.
\end{defn}

Notice, even though a rough Riemannian metric is not smooth, for any smooth function $w$ we can still define the gradient of $w$ with respect to $h$ by 
\begin{align}
  \nabla ^h w = h^{ij} \frac{\partial w}{\partial x^i}\frac{\partial}{\partial x^j}.   
\end{align}

Although the usual definition of distance function for Riemannian manifolds does not work for rough Riemannian manifolds since curves are measure zero and the metric is not required to be defined on sets of measure zero. Despite this fact, one can give the following definition of distance function.

\begin{defn}
    Let $g$ be a rough Riemannian metric on the smooth manifold $M$ and define the distance function
    \begin{align}
        d_g(x,y)=\sup\{w(y)-w(x): w \in C^{0,1}(M), |\nabla^g w|_g \le  1 \, a.e.\}.
    \end{align}
\end{defn}
It was observed in \cite{BandaraHassan} that this distance function induces a length space structure endowed with a Radon measure. Many more important properties of rough Riemannian metrics which are not needed here are studied in \cite{BandaraRough, BandaraBryan, BNR, BandaraHassan} which the interested reader may want to investigate. We finish by recalling a lemma which will be used frequently in the estimates given in section \ref{sec-Poincare}.

\begin{lem}\label{lem-RRMPreSetup}
    Let $M$ be a manifold and $h$ a Rough Riemannian metric. If we let $\eta:I\rightarrow M$ be a curve and $w \in C^{0,1}(M)$ and $|\nabla^h w|_h \le 1$ a.e.  with respect to $dV_h$ then $\frac{d}{dt} \left(w\circ \eta \right)\in L^{\infty}(M,g_0)$ and
    \begin{align}
        \left|\frac{d}{dt} \left(w\circ \eta \right)\right|  \le |\eta'|_h.
    \end{align}
\end{lem}
\begin{proof}
     If we let $\eta:I\rightarrow M$ be a curve then $\frac{d}{dt} \left(w\circ \eta \right)$ is in $L^{\infty}(M,g_0)$ since $w \in C^{0,1}(M,g_0)$. Now we can do a similar calculation as Proposition 4.3 of \cite{BandaraRough} to find that
       \begin{align}
          \left|\frac{d}{dt} \left(w\circ \eta \right)\right| &= |h(\nabla^h w, \eta')| \le |\nabla^h w|_h |\eta'|_h \le |\eta'|_h ,
       \end{align}
       where we use that $|\nabla^h w|_h \le 1$ a.e. in the last step.
\end{proof}

\section{Examples}\label{sec-Examples}
In this section, we present examples of rough Riemannian metrics to demonstrate that the hypotheses of Theorem \ref{MainTheoremConv} are necessary.  These examples are adaptation of examples which appear in previous works on Riemannian metrics.  In the setting of Riemannian metrics, an interpolating function $h_j$ is used to guarantee that the metrics are smooth.  In the Rough Riemannian setting, the interpolating function is removed, resulting in points of non-differentiability.

We begin with an adaptation of \cite[Example 3.1]{Allen-Bryden-Second} which shows that the dimension of the singular set in Theorem \ref{MainTheoremConv} is sharp for $p> 1$.

\begin{ex}\label{Ex: Dim of Singular Set is Sharp } 
Let $(\theta_1,...,\theta_m)$ be coordinates on $\Tor^m$, $n \in \N$, and define
\begin{align}
    \Sigma^n=\{(\theta_1,...,\theta_m) \in \Tor^m: \theta_{n+1}=\theta_{n+2}=...=\theta_m=0\}.
\end{align}
Let $r$ denote the distance to $\Sigma^{m-p}$ with respect to the metric $d\theta_1^2+\dots+ d\theta_m^2$ and consider the sequence of functions on $\Tor^m$, $p \in \N$, $p \not = 1$ given as follows:
\begin{equation}
f_j(r)=
\begin{cases}
\frac{j^{2}}{1+2\ln(j)} &\text{ if } r \in [0,1/j^{2}]
\\ \frac{1}{r(1-\ln(r))} &\text{ if } r \in (1/j^{2},1/j]
\\ 1 & \text{ otherwise } 
\end{cases}
\end{equation}
where $\eta > 1$. Then $M_j = (\Tor^m, f_j^2 g_{\Tor^m})$ converges to $\Tor^m$ in $L^q(\Tor^m)$, $\forall 1< q \le p$ and the distance function converge pointwise
\begin{align}
    d_j(x,y) \rightarrow d_{\Tor^m}(x,y), \quad \forall x,y \in \Tor^m \setminus \Sigma^{m-p},
\end{align}
but the distance function $d_j(x, \cdot)$, $x \in \Sigma^{m-p}$ does not converge pointwise.
\end{ex}

Now we see an example which shows that a subsequence is necessary in Theorem \ref{MainTheoremConv} which is exactly Example 3.2 of \cite{Allen-Bryden-Second}.

\begin{ex}\label{Ex: Subsequence is Needed}
Define 
\begin{align}
    Q_k= \left\{ \lp\frac{n}{2^{k+1}}, \frac{m}{2^{k+1}}\rp: n,m = 1,...,2^k\right\},
\end{align}
so that $Q_k$ collects dyadic points at scale $2^{-(k+1)}$ in $(0,1)\times (0,1)$.
Define a sequence of points $x_j$ by listing the points in $Q_1$ first, then the points in $Q_2$, then $Q_3$, etc. For each point $x_j$, choose radius $r_j=\frac{1}{2^{k-2}}$ where $x_j$ belonged to the set $Q_k$.
Define 
\begin{align}
f_j(r)=
    \begin{cases}
    r_j^{-1} & \text{ on } B_{r_j}(x_j)
    \\ 1 &\text{ elsewhere}
    \end{cases},
\end{align} 
and observe that $f_j \rightarrow 1$ in $L^p(\Tor^2)$,  $1 \le p < 2$.  In contrast, consider $M_j=(\Tor^2,g_j=f_j^2g_0)$ so that  for all $x,y \in M$, $d_j(x,y)$ does not converge pointwise to $d_{\Tor^2}(x,y)$ but that on a subsequence we find
\begin{align}
    d_k(x,y) \rightarrow d_{\Tor^2}(x,y), \quad \forall x,y \in \Tor^2\setminus \{(1/2,1/2)\}.
\end{align}

\end{ex}

We now review an example given by the first named author and C. Sormani in \cite{Allen-Sormani-2} which shows that the lower bound on distances assumed in Theorem \ref{MainTheoremConv} is necessary to converge to a distance function for a Riemannian metric.
 \begin{ex} \label{Ex Cinched-Sphere}  
Define functions radially from the north pole of a sphere $\Sp^m$ by  
 \be
 f_j(r)=
 \begin{cases}
 1 & r\in[0,\pi/2- 1/j]
 \\  h_0 & r\in[\pi/2- 1/j, \pi/2+ 1/j]
 \\ 1 &r\in [\pi/2+ 1/j, \pi]
 \end{cases}
\ee
where $h_0\in (0,1)$. Then $M_j = (\Sp^m, f_j^2 g_{\Sp^m})$
\begin{align}
d_j(x,y) \rightarrow d_{\infty}(x,y), 
\end{align}
uniformly for all $x,y \in M$ but $M_{\infty}$ is not isometric to $\Sp^m$. Instead $M_{\infty} = (\Sp^m, f_{\infty}^2 g_{\Sp^m}) $ is the Rough Riemannian metric with conformal factor
 \be
 f_{\infty}(r)=
 \begin{cases}
 h_0 & r=\pi/2
 \\  1 &\text{ otherwise}
 \end{cases}.
 \ee
 The distances between pairs of points near the equator in this limiting metric space is achieved by geodesics which run to the
 equator, and then around inside the cinched equator, and then out again.  
\end{ex}



See \cite{LehmanUndergrad2} for additional examples of rough Riemannian manifolds satisfying convergence properties.

\section{Sobolev Inequalities Relating Rough Riemannian Metrics to their Distance Functions}\label{sec-Poincare}

We start this section by adapting the potential estimate of \cite{Allen-Bryden-Second} to the Rough Riemannian case which will ultimately be used to establish the Sobolev inequality of Theorem \ref{Main Sobolev Ineq}. The main difference in the proof is the need to use a Lipschitz function which almost achieves the supremum defining the distance function for Rough Riemannian metrics as opposed to the length of a curve which achieves the distance for a smooth Riemannian metric.

\begin{lem}\label{lem: distance potential estimate}
    Let $\varepsilon>0$ and $(M,g_{0})$ be some fixed $m$-dimensional Riemannian manifold with metric $g_{0}$, and let $h$ be a Rough Riemannian metric on $M$. Finally, let $d_{h}$ be the distance function determined by $h$. Then, we have
    \begin{equation}
        d_h(x,y)\le C(M,g_0,\varepsilon)\left(\int_{M}\frac{|h|^{\frac{1}{2}}_{g_{0}}dV_{g_0}(z)}{d_{g_{0}}(x,z)^{m-1}}+\int_{M}\frac{|h|^{\frac{1}{2}}_{g_{0}}dV_{g_0}(z)}{d_{g_{0}}(y,z)^{m-1}}\right).
    \end{equation}
\end{lem}
\begin{proof}
    Fix $\varepsilon>0$ and let $U_{\vare}$ be a symmetric family of curves with width $\vare$ connecting $x$ to $y$, and let $\Gamma$ be the $m-1$ Euclidean ball of radius $\vare$ which parameterizes the curves foliating $U_{\vare}$. Given $\delta>0$, fix $x,y \in M$ and let $w_{x,y}(z):=w(z)$ be such that $d_h(x,y) \le |w(x)-w(y)|+ \delta$ so that $w \in C^{0,1}(M)$ and $|\nabla^h w|_h \le 1$ a.e. Then we can calculate
    \begin{align}
       \omega_{m-1}\vare^{m-1} d_{h}(x,y)&\le\omega_{m-1}\vare^{m-1} |w(x)-w(y)|+\delta\omega_{m-1}\vare^{m-1}
       \\&=\int_{\Gamma_{\varepsilon}}\left|\int_{[0,L]} \frac{d}{dt} \left(w\circ \eta \right)dt\right| d \mu+\delta\omega_{m-1}\vare^{m-1},
        \\&\le\int_{\Gamma_{\varepsilon}}\int_{[0,L]} \left|\frac{d}{dt} \left(w\circ \eta \right)\right|dt d \mu+\delta\omega_{m-1}\vare^{m-1},
       \end{align}
       where the last integral is well defined by Lemma \ref{lem-RRMPreSetup}.
       Continuing the calculation and using Lemma \ref{lem-RRMPreSetup} we find
       \begin{align}
\omega_{m-1}\vare^{m-1} d_{h}(x,y)&\leq\int_{\Gamma_{\varepsilon}}\int_{[0,L]}|\eta'|_hdtd \mu+\delta\omega_{m-1}\vare^{m-1}
       \\&\le\int_{\Gamma_{\varepsilon}}\int_{[0,L]}|h|_{g_{0}}^{\frac{1}{2}}|\eta'|_{g_{0}}+\frac{\delta}{L}dtd \mu,\label{EqImportant}
    \end{align}
    using Lemma \ref{NormComparisonLemma} in the last step.
    
     Plugging \eqref{EqImportant} into Lemma \ref{Coarea Formula to Potential Integral}, using Remark \ref{remark-TrivialTransformSubstitution}, and using the fact that $|\eta'|_{g_0}$ is uniformly bounded above from Lemma \ref{MainToolsForProof}, we get

     \begin{align}
        d_{h}(x,y)&\leq  C'(M,g_{0})\left(\int_{M}\frac{|h|^{\frac{1}{2}}_{g_{0}}dV_{g_0}(z)}{d_{g_{0}}(x,z)^{m-1}}+\int_{M}\frac{|h|^{\frac{1}{2}}_{g_{0}}dV_{g_0}(z)}{d_{g_{0}}(y,z)^{m-1}}+\frac{\delta}{L}\right).
    \end{align}

    Hence, by letting $\delta \rightarrow 0$ we find the desired conclusion
    \begin{align}
        d_{h}(x,y)
        \leq  C(M,g_{0})\left(\int_{M}\frac{|h|^{\frac{1}{2}}_{g_{0}}dV_{g_0}(z)}{d_{g_{0}}(x,z)^{m-1}}+\int_{M}\frac{|h|^{\frac{1}{2}}_{g_{0}}dV_{g_0}(z)}{d_{g_{0}}(y,z)^{m-1}}\right).
    \end{align}
\end{proof}

The next lemma adapts the proof of the previous lemma in the case that is most important for Theorem \ref{MainTheoremConv}. Since the distance function of the difference of two Riemmanian metrics is not equal to the difference of the distance functions, some care is needed in establishing the potential estimate for $d_h-d_{g_0}$ in terms of the potential applied to $|h-g_0|_{g_0}$.

\begin{lem}\label{lem: distance potential estimate with error}
    Let $(M,g_{0})$ be a Riemannian manifold, and let $h$ be a Rough Riemannian metric on $M$ such that $h(v,v)\geq g_{0}(v,v)$ for all $v \in T_pM$. Then there exists a constant $C(M,g_0,\vare)>0$ such that if $U_{\vare}$ is any symmetric family of curves of width $\vare$ connecting two points, say $x$ and $y$, then we have
    \begin{equation*}
        d_{h}(x,y)-d_{g_{0}}(x,y)-\vare\le C\left(\int_{M}\frac{|h-g_{0}|^{\frac{1}{2}}_{g_{0}}dV_{g_0}(z)}{d_{g_{0}}(x,z)^{m-1}}+\int_{M}\frac{|h-g_{0}|^{\frac{1}{2}}_{g_{0}}dV_{g_0}(z)}{d_{g_{0}}(y,z)^{m-1}}\right),
    \end{equation*}
    where $C=C(M,g_0,\vare)\rightarrow\infty$ as $\vare\rightarrow0$.
\end{lem}
\begin{proof}
    Given $\vare>0$ we can use Lemma \ref{uniform distance estimate symmetric family of curves} to pick $\delta(\vare)=\delta>0$ such that if $U_{\delta}$ is a symmetric family of curves connecting any two points, say $x$ and $y$, then we have
    \begin{equation}
        \int_{[0,L]}|\Dot{\eta}|_{g_{0}}dt\leq d_{g_{0}}(x,y)+\vare,
    \end{equation}
    where $\eta$ is any of the curves foliating $U_{\delta}$. 

    Given $\lambda>0$, fix $x,y \in M$ and let $w_{x,y}(z):=w(z)$ be such that $d_h(x,y) \le |w(x)-w(y)|+ \lambda$ so that for $\nabla ^h w = h^{ij} \frac{\partial w}{\partial x^i}\frac{\partial}{\partial x^j}$ defined in coordinates we see that $w \in C^{0,1}(M)$ and $|\nabla^h w|_h \le 1$ a.e. Using the fundamental theorem of calculus we find    
    \begin{align}
       &\omega_{m-1}\delta^{m-1}( d_{h}(x,y)-d_{g_{0}}(x,y)-\vare)
        \\& =  \omega_{m-1}\delta^{m-1}( |w(y)-w(x)|+\lambda-d_{g_{0}}(x,y)-\vare)
       \\ &\leq\int_{\Gamma_{\delta}}\left|\int_0^L\frac{d}{dt}(w \circ \eta)(t)dt\right|-\int_0^L\sqrt{g_{0}(\Dot{\eta},\Dot{\eta})}dt\mu_{\delta}+\omega_{m-1}\delta^{m-1}(\lambda - \varepsilon)
       \\ &\leq\int_{\Gamma_{\delta}}\int_0^L\left|\frac{d}{dt}(w \circ \eta)(t)\right|-\sqrt{g_{0}(\Dot{\eta},\Dot{\eta})}dt\mu_{\delta}+\omega_{m-1}\delta^{m-1}\lambda 
       \end{align}
        where the last integral is well defined by Lemma \ref{lem-RRMPreSetup}. Continuing the calculation using Lemma \ref{lem-RRMPreSetup} we find
       \begin{align}
       &\int_{\Gamma_{\delta}}\int_0^L\left|\frac{d}{dt}(w \circ \eta)(t)\right|-\sqrt{g_{0}(\Dot{\eta},\Dot{\eta})}dt\mu_{\delta}+\omega_{m-1}\delta^{m-1}\lambda
       \\&\leq\int_{\Gamma_{\delta}}\int_0^L\sqrt{h(\Dot{\eta},\Dot{\eta})}-\sqrt{g_{0}(\Dot{\eta},\Dot{\eta})} + \frac{\lambda}{L}dt\mu_{\delta}
        \\
        &\le\int_{\Gamma_{\delta}}\int_0^L\sqrt{(h-g_{0})(\Dot{\eta},\Dot{\eta})}+ \frac{\lambda}{L}dt\mu_{\delta},\label{Eq-LastofGroup}
    \end{align}
    which holds since, by assumption, $|h|_{g_{0}}\geq |g_{0}|_{g_{0}}$. Now, we may estimate the first term in \eqref{Eq-LastofGroup} by
    \begin{equation}
        \int_{\Gamma_{\delta}}\int_{[0,L]}\sqrt{(h-g_{0})(\Dot{\eta},\Dot{\eta})}dt\mu_{\delta}\leq\int_{\Gamma_{\delta}}\int_{[0,L]}|h-g_{0}|_{g_{0}}^{\frac{1}{2}}|\Dot{\eta}|_{g_{0}}dt\mu_{\delta}.
    \end{equation}
    Thus, we have
    \begin{equation}
        \omega_{m-1}\delta^{m-1}( d_{h}(x,y)-d_{g_{0}}(x,y)-\vare)\leq \int_{\Gamma_{\delta}}\int_{[0,L]}|h-g_{0}|_{g_{0}}^{\frac{1}{2}}|\Dot{\eta}|_{g_{0}}+ \frac{\lambda}{L}dt\mu_{\delta}.\label{Eq-Important 2}
    \end{equation}
    
     Plugging \eqref{Eq-Important 2} into Lemma \ref{Coarea Formula to Potential Integral}, using Remark \ref{remark-TrivialTransformSubstitution}, and using the fact that $|\eta'|_{g_0}$ is uniformly bounded above from Lemma \ref{MainToolsForProof}, we find
    \begin{equation}
        \omega_{m-1}\delta^{m-1}\left(d_{h}(x,y)-d_{g_{0}}(x,y)-\vare\right)
    \end{equation}
    is bounded by
    \begin{equation*}
        C\left(M,g_{0}\right)\left(\int_{M}|h-g_{0}|^{\frac{1}{2}}_{g_0}\left(d_{g_{0}}(x,y)^{1-m}+d_{g_{0}}(y,z)^{1-m}\right)dV_{g_0}(z)+\frac{\lambda}{L}\right).
    \end{equation*}
    So, by letting $\lambda \rightarrow 0$ and then dividing both sides by $\omega_{m-1}\delta^{m-1}(\vare)$ we find
    \begin{align}
         d_{h}(x,y)-d_{g_{0}}(x,y)-\vare\leq& C(M,g_0)\delta(\vare)^{1-m}\int_{M}\frac{|h-g_{0}|^{\frac{1}{2}}_{g_{0}}dV_{g_0}(z)}{d_{g_{0}}(x,z)^{m-1}}
         \\
         &+C(M,g_0)\delta(\vare)^{1-m}\int_{M}\frac{|h-g_{0}|^{\frac{1}{2}}_{g_{0}}dV_{g_0}(z)}{d_{g_{0}}(y,z)^{m-1}}.     
    \end{align}
    Setting $C(M,g_0,\varepsilon):=C(M,g_0)\delta(\vare)^{1-m}$ we find the desired result.
\end{proof}
We now turn the potential estimates above into integral estimates. The first half of the theorem is a restatement of Theorem \ref{Main Sobolev Ineq} and the second half of the theorem is used to prove Theorem \ref{MainTheoremConv}.
\begin{thm}\label{thm-potentialToIntegral}
    Let $(M,g_{0})$ be a fixed Riemannian manifold with metric $g_{0}$, and let $h$ be a Rough Riemannian metric on $M$. Then, we have
    \begin{equation}
        \|d_{h}(x,y)\|_{L^{q}(M\times M, g_{0}\oplus g_{0})}\le C(M,g_0,q)\|h\|^{\frac{1}{2}}_{L^{\frac{p}{2}}(M,g_{0})},
    \end{equation}
    with $q<\frac{np}{n-p}$.
    
    Suppose $\varepsilon>0$, $h\geq g_{0}$, and let $f_{\vare}(x,y)=\max\lbrace0,d_{h}(x,y)-d_{g_{0}}(x,y)-\vare\rbrace$. Then, we have
    \begin{equation}
        \|f_{\vare}(x,y)\|_{L^{q}(M\times M,g_{0}\oplus g_{0})}\leq C(\vare,M,g_0,q)\|h-g_{0}\|^{\frac{1}{2}}_{L^{\frac{p}{2}}(M,g_{0})},
    \end{equation}
    with $q<\frac{np}{n-p}$.
\end{thm}
\begin{proof}
    We prove only the first inequality, since the proof of the second is nearly identical using Lemma \ref{lem: distance potential estimate with error} instead of Lemma \ref{lem: distance potential estimate}. By applying Lemma \ref{lem: distance potential estimate} we can calculate that
    \begin{equation}
        \|d_{h}\|^q_{L^{q}}\leq C(\vare)\int_{M\times M}\left(\int_{M}\frac{|h|^{\frac{1}{2}}_{g_{0}}dV_{g_0}(z)}{d_{g_{0}}(x,z)^{m-1}}+\int_{M}\frac{|h|^{\frac{1}{2}}_{g_{0}}dV_{g_0}(z)}{d_{g_{0}}(y,z)^{m-1}}\right)^{q}dV_{g_0}(x)dV_{g_0}(y).
    \end{equation}
    The term on the right above may be estimated by
    \begin{equation}
        C(\vare,q)\int_{M\times M}\left(\int_{M}\frac{|h|^{\frac{1}{2}}_{g_{0}}dV_{g_0}(z)}{d_{g_{0}}(x,z)^{m-1}}\right)^{q}+\left(\int_{M}\frac{|h|^{\frac{1}{2}}_{g_{0}}dV_{g_0}(z)}{d_{g_{0}}(y,z)^{m-1}}\right)^qdV_{g_0}(x)dV_{g_0}(y).
    \end{equation}
    Next, we see that the above is bounded by
    \begin{align}
        &\vol_{g_{0}}(M)\int_{M}\left(\int_{M}\frac{|h|^{\frac{1}{2}}_{g_{0}}dV_{g_0}(z)}{d_{g_{0}}(x,z)^{m-1}}\right)^qdV_{g_0}(x)
        \\&\qquad +\vol_{g_{0}}(M)\int_{M}\left(\int_{M}\frac{|h|^{\frac{1}{2}}_{g_{0}}dV_{g_0}(z)}{d_{g_{0}}(y,z)^{m-1}}\right)^qdV_{g_0}(y)
        \\&=2\vol_{g_0}(M)\|(V_{g_0}|h|_{g_0}^{\frac{1}{2}})\|_{L^q}^q,
    \end{align}
    and so by applying Proposition \ref{PotentialOperatorEstimate} when $q < \frac{m-p}{mp}$ we see that
    \begin{equation}
        \|d_{h}(x,y)\|_{L^{q}(M\times M, g_{0}\oplus g_{0})}\le C(M,g_0,q,\varepsilon)\|h\|^{\frac{1}{2}}_{L^{\frac{p}{2}}(M,g_{0})},
    \end{equation}
    where we are allowed to fix an $\varepsilon$ in order to remove the dependence on $\varepsilon$ in the constant. When repeating this argument for $f_{\varepsilon}$, the $\varepsilon$ dependence remains throughout the inequality.
\end{proof}

Using the Theorem \ref{thm-potentialToIntegral} we can now establish the convergence result of Theorem \ref{MainTheoremConv}. The key to the proof is applying Lemma \ref{Lp Hausdorff dimension control} which establishes that the Hausdorff dimension of the set on which the integral on the right hand side of Theorem \ref{thm-potentialToIntegral} is not well controlled is less than $m$. Hence the set where the distance function does not converge pointwise will have measure zero.

\begin{thm}
Let $M$ be a connected, closed, and oriented manifold, $g_0$ a smooth Riemannian manifold and $g_j$ a sequence of Rough Riemannian manifolds. If
\begin{align}
    g_j(v,v) \ge \left(1-\frac{1}{j} \right) g_0(v,v), \quad \forall p\in M, v \in T_pM,
\end{align}
and
\begin{align}
    \|g_j-g_0\|_{L^p_{g_0}(M)}\rightarrow 0,\quad p\ge\frac{1}{2},
\end{align}
then there exists a subsequence $g_k$ and there exists $U \subset M$ measurable so that $\vol_{g_0}(U)=\vol_{g_0}(M)$  and $\mathrm{dim}_{H_{g_0}}(M \setminus U)\le m-2p$ and
\begin{align}
    d_k(x,y) \rightarrow d_0(x,y), \quad \forall x,y \in U.
\end{align}
Furthermore, for every $\tau > 0$ we can find a measurable set $U_{\tau}\subset M$ so that $\vol_{g_0}(M \setminus U_{\tau})< \tau$ and 
\begin{align}
    d_k(x,y) \rightarrow d_0(x,y),
\end{align}
uniformly $\forall x,y \in U_{\tau}$.
\end{thm}
\begin{proof}
    Since $\|g_{j}-g_{0}\|_{L^{p}_{g_{0}}(M)}\rightarrow 0$, we can find a sub-sequence $g_k$ so that 
    \begin{equation}
        \sum^{\infty}_{k=1}\|g_{k}-g_{0}\|^p_{L^{p}_{g_{0}}(M)}<\infty.
    \end{equation}
    Now, let $f_{k}=\sqrt{|g_{k}-g_{0}|_{g_{0}}}$ a.e. Thus the sequence $f_{k}$ satisfies the hypothesis of Proposition \ref{Lp Hausdorff dimension control}. Therefore, for any $\delta>0$, we have that $\mathrm{dim}_{H_{g_0}}B(\delta)\le m-2p$.
    
    Let $D=\bigcup^{\infty}_{n=1} B\left(\frac{1}{n}\right)$. Since the Hausdorff dimension of each $B(\frac{1}{n})$ is bounded above by $m-2p$, it follows that the Hausdorff dimension of $D$ is less than $m-2p$. Let $x$ and $y$ be in $U:=M\setminus D$. Then, by the definition of $D$, we see that for $x,y$ we have
    \begin{equation}
        \lim_{k\rightarrow\infty}\int_M\frac{f_{k}(z)}{d_{g_0}(x,z)^{m-1}}+\frac{f_{k}(z)}{d_{g_0}(y,z)^{m-1}}dV_{g_0}(z)=0.
    \end{equation}
    From Lemma \ref{lem: distance potential estimate with error}, we have for any $\vare>0$ that
    \begin{equation}
        d_{g_{k}}(x,y)-d_{g_{0}}(x,y)-\vare\le C(\vare)\left(\int_M\frac{f_{k}(z)}{d_{g_0}(x,z)^{m-1}}+\frac{f_{k}(z)}{d_{g_0}(y,z)^{m-1}}dV_{g_0}(z)\right).
    \end{equation}
    Since $x$ and $y$ are in $U=M\setminus D$, it follows that the right hand side of the above tends towards zero as $k\rightarrow\infty$. Thus, for each $x$ and $y$ in $U$, and each $\vare>0$, we have
    \begin{equation}
        \limsup_{k\rightarrow\infty}d_{g_{k}}(x,y)\le d_{g_{0}}(x,y)+\vare.
    \end{equation}
    Finally, we let $\vare$ tend to zero. Thus, for $x$ and $y$ in $U$, we have
    \begin{equation}
        \limsup_{k\rightarrow\infty}d_{g_{k}}(x,y)\le d_{g_0}(x,y).
    \end{equation}
    However, the distance below estimate for each $g_k$ implies that $\left(1-\frac{1}{k}\right)d_{g_0} \le d_{g_k}$ on $M$ which implies that we in fact have $\displaystyle\lim_{k \rightarrow \infty} d_{g_k}(x,y)=d_{g_0}(x,y)$.
    Then by Egeroff's theorem, for any $\tau>0$ we can find a set $U_{\tau}$ so that $ U_{\tau} \subset M$, $\vol(M \setminus U_{\tau})<\tau$ and $d_k$ converges uniformly on $U_{\tau}$, as desired.
\end{proof}

We finish with a proof of Theorem \ref{MainThmEst} where again the symmetric family of curves is exploited to obtain a H\"{o}lder estimate on the distance function. One can key into the importance of $p>m$ in the proof for concluding that the integral of the distortion of volume of the symmetric family of curves is bounded.

\begin{proof}[Proof of Theorem \ref{MainThmEst}]

Given $\vare>0$ we can use Lemma \ref{uniform distance estimate symmetric family of curves} to pick $\delta(\vare)=\delta>0$ such that if $U_{\delta}$ is a symmetric family of curves connecting any two points, say $x$ and $y$, then we have
    \begin{equation}
        \int_{[0,L]}|\Dot{\eta}|_{g_{0}}dt\leq d_{g_{0}}(x,y)+\vare,
    \end{equation}
    where $\eta$ is any of the curves foliating $U_{\delta}$. 

    Given $\lambda>0$, fix $x,y \in M$ and let $w_{x,y}(z):=w(z)$ be such that $d_h(x,y) \le |w(x)-w(y)|+ \lambda$ so that for $\nabla ^h w = h^{ij} \frac{\partial w}{\partial x^i}\frac{\partial}{\partial x^j}$ defined in coordinates we see that $w \in C^{0,1}(M)$ and $|\nabla^h w|_h \le 1$ a.e. Using the fundamental theorem of calculus we find    
\begin{align}
\vol(\Gamma_{\delta})& d_{h}(x,y)
\\&\le \vol(\Gamma_{\delta}) (|w(x)-w(y)|+\lambda)
\\&\le \int_{\Gamma_{\delta}} \left|\int_0^L\frac{d}{dt}(w \circ \eta)(t)dt\right| d\mu_{\delta}d\tau +\lambda \vol(\Gamma_{\delta})
\\&\le \int_{\Gamma_{\delta}} \int_0^L\left|\frac{d}{dt}(w \circ \eta)(t)\right|dt d\mu_{\delta}d\tau +\lambda \vol(\Gamma_{\delta}).
\end{align}
Continuing the calculation using Lemma \ref{lem-RRMPreSetup}, allowing $\lambda \rightarrow 0$ since $w$ no longer appears, we find
       \begin{align}
       &\vol(\Gamma_{\delta}) d_{h}(x,y)
\\&\le \int_{\Gamma_{\delta}} \int_0^L \sqrt{h(\eta',\eta')} dt d \mu_{\delta}
\\&\le \int_{\Gamma_{\delta}} \int_0^L |h|_{g_0}^{\frac{1}{2}}|\eta'|_{g_0} dt d \mu_{\delta}
\\&\le C(M,g_0,\delta) L^{1-m} \int_{U_{\varepsilon}} |h|_{g_0}^{\frac{1}{2}}\lp\sin\lp\frac{\pi}{L}t\rp\rp^{1-m}dV_{g_0}. \label{WEpsilonMiddle}
\end{align}

By moving $\vol(\Gamma_{\delta})$ to the other side and using Holder's inequality with $p>1$, $q = \frac{p}{p-1}$
\begin{align}
 d_{h}(x,y)
 &\le C(M,g_0,\delta)\vol(\Gamma_{\delta})^{-1}L^{1-m}  \lp\int_{U_{\delta}} |h|_{g_0}^{\frac{p}{2}}dV_{g_0}\rp^{1/p} 
 \\&\quad \cdot\lp \int_{U_{\delta}}\sin\lp\frac{\pi}{L}t \rp^{-\frac{p(m-1)}{p-1}} dV_{g_0} \rp^{\frac{p-1}{p}}.\label{WEpsilonEnd}
\end{align}

By using the coarea formula again we see that
\begin{align}
&\lp\int_{U_{\delta}}\sin\lp\frac{\pi}{L}t \rp^{-\frac{p(m-1)}{p-1}} dV_{g_0} \rp^{\frac{p-1}{p}} 
\\&\le C(M,g_0,\delta)L^{\frac{(m-1)(p-1)}{p}}\lp\int_0^L \sin\lp\frac{\pi}{L}t\rp ^{-\frac{p(m-1)}{p-1}+m-1}  dt\rp^{\frac{p-1}{p}}
\end{align}
where one can see the details worked out in \cite{Allen}.

So if we choose $p > m$ then the integral of the power of $\sin$ is integrable and by changing variables we find
\begin{align}
&\lp\int_0^L \sin\lp\frac{\pi}{L}t\rp ^{-\frac{p(m-1)}{p-1}+m-1}  dt\rp^{\frac{p-1}{p}} 
\\& =L^{\frac{p-1}{p}}\lp\int_0^1 \sin\lp\pi s\rp ^{-\frac{p(m-1)}{p-1}+m-1}  ds\rp^{\frac{p-1}{p}}\le L^{\frac{p-1}{p}}\tilde{C}_{p}.\label{TrigIntegrable}
\end{align}

Furthermore, by collecting all powers of $L$ we know that
\begin{align}
&\frac{p-1}{p}+\frac{(m-1)(p-1)}{p}+1-m=\frac{p-m}{p},\label{FinalPower}
\end{align}

Now by choosing $\delta$, which can be done uniformly for any $x,y \in M$ by Lemma \ref{NorNeighRadiusEst}, we know that $C(M,g_0,\delta)\le C(M,g_0)$.
Putting \eqref{WEpsilonEnd}, \eqref{TrigIntegrable}, and \eqref{FinalPower} together we find 
\begin{align}
d_{h}(x,y) \le L^{\frac{p-m}{p}}\tilde{C}(M,g_0) \lp\int_{U_{\delta}} |h|_{g_0}^{\frac{p}{2}}dV_{g_0}\rp^{1/p}.\label{OneFromTheEnd}
\end{align}

Then by substituting $L = d_{g_0}(q_1,q_2)$ in \eqref{OneFromTheEnd} and using the fact that
\begin{align}
\lp\int_{U_{\varepsilon}} |h|_{g_0}^{\frac{p}{2}}dV_{g_0}\rp^{1/p} \le \lp\int_{M} |h|_{g_0}^{\frac{p}{2}}dV_{g_0}\rp^{1/p}
\end{align}
we find the desired bound on distance.
\end{proof}

\bibliographystyle{alpha}
\bibliography{bibliography}
\end{document}